\documentclass{article}
\usepackage{amsmath, amsthm}
\usepackage{amssymb}
\usepackage{textcomp}
\newtheoremstyle{theorem}
  {10pt}          
  {10pt}  
  {\sl}  
  {\parindent}     
  {\bf}  
  {. }    
  { }    
  {}     
\theoremstyle{theorem}
\newtheorem{theorem}{Theorem}[section]

\newtheorem{lemma}[theorem]{Lemma}
\newtheorem{remark}[theorem]{Remark}
\newtheorem{proposition}[theorem]{Proposition}
\newtheorem{example}[theorem]{Example}
\newtheoremstyle{defi}
  {10pt}          
  {10pt}  
  {\rm}  
  {\parindent}     
  {\bf}  
  {. }    
  { }    
  {}     
\theoremstyle{defi}
\newtheorem{definition}[theorem]{Definition}

\begin{document}

\date{}
\author{Thabet ABDELJAWAD\footnote{\c{C}ankaya University, Department of
Mathematics, 06530, Ankara, Turkey} , Duran T\"{U}RKO\~{G}LU \footnote{ Department of Mathematics, Faculty of Science and Arts, Gazi University, 06500, Ankara-Turkey. dturkoglu@gazi.edu.tr.}}

\title{Locally Convex Valued Rectangular Metric Spaces and The Kannan's Fixed Point Theorem }
\maketitle

\begin{abstract}
Rectangular TVS-cone   metric spaces are introduced and Kannan's fixed point theorem is proved in these spaces. Two approaches are followed for the proof. At first we prove the theorem by a direct method using the structure of the space itself. Secondly, we use the nonlinear scalarization used recently by Wei-Shih Du in [A note on cone metric fixed point theory and its equivalence, {Nonlinear Analysis},72(5),2259-2261 (2010).] to prove the equivalence of the Banach contraction principle in cone metric spaces and usual metric spaces. The proof is done without any normality assumption on the cone of the locally convex topological vector space, and hence generalizing several previously obtained results.
\end{abstract}

\emph{Keywords}: TVS-cone metric space, rectangular TVS-cone metric space, Kannan's fixed point theorem.

\section{Introduction and Preliminaries } \label{s:1}

Many authors attempted to generalize the notion of the metric space.
 In 2007, Huang and Zhang \cite{HZ} announced the notion of cone metric spaces
(CMS) by using the same idea, namely, by replacing real numbers with an ordering real Banach space.
In that paper, they also discussed some properties of convergence of sequences and proved the
fixed point theorems of contractive mapping for cone metric spaces:
Any mapping $T$ of a complete cone metric space $X$ into itself that
satisfies, for some $0\leq k<1$, the inequality $d(Tx,Ty)\leq k
d(x,y)$, for all $x,y \in X$, has a unique fixed point. Lately, many
results on fixed point theorems have been extended to cone metric
spaces  (see e.g.\cite{HZ},\cite{RH},\cite{Ishak},\cite{TAA},\cite{TAA2},\cite{K},\cite{T},\cite{AK}, \cite{TA}). For Kannan's fixed point theorem in rectangular metric spaces (R-MS) we refer to \cite{Das} and for the contraction principle and Kannan's fixed point theorem in rectangular cone metric space (R-CMS) see \cite{Akbar} and \cite{Beg}, respectively.

Recently, Du \cite{D_2009} gave the definition of generalized cone
metric space, namely topological vector space-cone metric space
(TVS-CMS), and proved some fixed point theorems on that class.
The author showed also that Banach contraction principles in
usual metric spaces and in TVS-CMS are equivalent.

In this manuscript, we first introduce the notion of rectangular TVS-cone metric spaces (R-TVS-CMS) and then prove Kannan's fixed point theorem in this class of spaces.  The obtained result generalizes those in \cite{Beg} and \cite{Das} and hence the classical Kannan's fixed point theorem. Two proofs are presented and the proofs are done without any normality assumption.

Throughout this paper, $(E,S)$ stands for real Hausdorff locally convex topological vector space (t.v.s.) with $S$ its generating system of seminorms.
A non-empty subset $P$ of $E$ is called  cone if $P+P \subset P$, $\lambda P \subset P$
for $\lambda \geq 0$ and  $P \cap (-P) =\{0\}$. The cone $P$ will be assumed to be closed and has nonempty interior as well.
For a given cone $P$, one can define  a partial ordering (denoted by $\leq$ or $\leq_P$) with respect to $P$
by $x\leq y$ if and only if $y-x \in P$. The notation $x<y$ indicates that $x\leq y$ and $x\neq y$ while
$x<<y$ will show $y-x\in intP$, where $intP$ denotes the interior of $P$.
Continuity of the algebric operations in a topological vector space and the properties of the cone imply the relations:
$$intP+intP\subseteq intP ~\emph{and}~\lambda intP \subseteq intP~(\lambda > 0).$$
We appeal to these relations in the following.

\begin{definition} \cite{Ali}
A cone $P$ of a topological vector space $(X,\tau)$ is said to be normal  whenever $\tau$ has a base of zero consisting of $P-$ full sets. Where a subset of $A$ of an order vector space via a cone $P$ is said to be $P-$full if for each $x, y \in A$ we have $\{a \in E: x\leq a \leq y\}\subset A$.
\end{definition}
\begin{theorem} \cite{Ali}
(a) A cone $P$ of a topological vector space $(X,\tau)$ is normal if and only if whenever $\{x_\alpha\}$ and $\{y_\alpha\}$, $\alpha \in \Delta$ are two nets in $X$ with $0\leq x_\alpha \leq y_\alpha$ for each $\alpha \in \Delta$ and $y_\alpha \rightarrow 0$, then $x_\alpha \rightarrow 0$.

(b) The cone of an ordered locally convex space $(X,\tau)$ is normal if and only if $\tau$ is generated by a family of monotone $\tau-$ continuous seminorms. Where a seminorm $q$ on $X$ is called monotone if $q(x)\leq q(y)$ for all $x, y \in X$ with $0\leq x \leq y$.
\end{theorem}

In particular, if   $P$ is a cone of a real Banach space $E$, then it  is called
\textit{normal} if there is a number $K \geq 1$ such that for
all $x,y \in E$:\  $
0\leq x \leq y\Rightarrow \|x\|\leq K \|y\|.$ The least positive integer $K$, satisfying this inequality,
is called the normal constant of $P$.
Also, $P$ is said to be  \textit{regular} if every increasing sequence which is bounded
from above is convergent. That is, if $\{x_n\}_{n\geq 1}$ is a sequence such
that $x_1 \leq x_2\leq \cdots\leq y$ for some $y \in E$, then there is $x \in E$
such that $\lim_{n\rightarrow\infty} \|x_n-x\|=0$. For more details about cones in locally convex topological vector spaces we may refer the reader to \cite{Ali}.

\

\begin{definition} (See \cite{CHY}, \cite{D_2008}, \cite{D_2009})
For $e \in intP$, the nonlinear scalarization function  $\xi_e:E\rightarrow \mathbb R$ is defined by
\[\xi_e(y)=\inf\{t \in \mathbb R: y \in te-P\}, \ \mbox{for all} \ y \in E.\]
\end{definition}
\begin{lemma} (See \cite{CHY}, \cite{D_2008}, \cite{D_2009})
For each $t\in \mathbb R$ and $y \in E$, the following are satisfied:
\begin{itemize}
\item[$(i)$] $\xi_e(y)\leq t\Leftrightarrow  y \in te-P$,
\item[$(ii)$] $\xi_e(y)> t\Leftrightarrow  y \notin te-P$,
\item[$(iii)$] $\xi_e(y)\geq t\Leftrightarrow  y \notin te-intP$,
\item[$(iv)$] $\xi_e(y)< t\Leftrightarrow  y \in te-intP$,
\item[$(v)$] $\xi_e(y)$ is positively homogeneous and continuous  on $E$,
\item[$(vi)$] if $y_1\in y_2+P$, then $\xi_e(y_2)\leq \xi_e(y_1)$,
\item[$(vii)$] $\xi_e(y_1+y_2)\leq \xi_e(y_1)+\xi_e(y_2)$, for all $y_1,y_2 \in E$.
\end{itemize}
\label{lemma_scalarization}
\end{lemma}

\begin{definition} Let $X$ be a non-empty set and $E$ as usual a Hausdorff locally convex topological space. Suppose  a vector-valued function $p:X\times X\rightarrow E$ satisfies:
\begin{enumerate}
\item[$(M1)$] $0\leq p(x,y)$ for all $x,y \in X$,
\item[$(M2)$] $p(x,y)=0$ if and only if $x=y$,
\item[$(M3)$] $p(x,y)=p(y,x)$ for all $x,y \in X$
\item[$(M4)$] $p(x,y) \leq p(x,z)+p(z,y)$, for all $x,y,z \in X$.
\end{enumerate}
Then,  $p$ is called TVS-cone metric on $X$, and the pair $(X,p)$ is called
a  TVS-cone metric space (in short, TVS-CMS).
\end{definition}

Note that in \cite{HZ}, the authors considered $E$ as a real Banach space
in the definition of TVS-CMS. Thus, a cone metric space (in short, CMS) in the sense of Huang and
Zhang \cite{HZ} is a special case of TVS-CMS.

\begin{lemma} (See \cite{D_2009})
Let $(X,p)$ be a TVS-CMS. Then, $d_p:X \times X\rightarrow [0,\infty)$ defined by $d_p=\xi_e\circ p$ is a metric.
\label{lemma_usual_metric}
\end{lemma}

\begin{remark}
Since a cone metric space $(X,p)$ in the sense of Huang and Zhang \cite{HZ}, is a special case of TVS-CMS, then
$d_p:X \times X\rightarrow [0,\infty)$ defined by $d_p=\xi_e\circ p$ is also a metric.
\label{remark_CMS_usual_ms}
\end{remark}

\begin{definition}(See \cite{D_2009})
Let $(X,p)$ be a TVS-CMS, $x\in X$ and $\{x_n\}_{n=1}^{\infty}$ a sequence in $X$.
\label{definition_convergence}
\begin{itemize}
\item[($i$)] $\{x_n\}_{n=1}^{\infty}$ TVS-cone converges to $x\in X$ whenever
for every $0<<c\in E$, there is a natural number $M$ such that $p(x_n,x)<<c$
for all $n\geq M$ and denoted by $cone-\lim_{n\rightarrow \infty}x_n=x$ (or
$x_n\stackrel{cone}{\rightarrow} x$ as $n\rightarrow \infty$),
\item[($ii$)] $\{x_n\}_{n=1}^{\infty}$ TVS-cone Cauchy sequence in $(X,p)$ whenever
for every $0<<c\in E$, there is a natural number $M$ such that $p(x_n,x_m)<<c$
for all $n,m \geq M$,
\item[($iii$)]  $(X,p)$ is TVS-cone complete if every sequence TVS-cone Cauchy sequence  in $X$ is a TVS-cone
convergent.
\end{itemize}
\end{definition}

\begin{lemma} (See \cite{D_2009})
Let $(X,p)$ be a TVS-CMS, $x\in X$ and $\{x_n\}_{n=1}^{\infty}$ a sequence in $X$.
Set $d_p=\xi_e \circ p$. Then the following statements hold:
\begin{itemize}
\item[($i$)] If $\{x_n\}_{n=1}^{\infty}$ converges to $x$ in TVS-CMS $(X,p)$, then
$d_p(x_n,x)\rightarrow 0$ as $n\rightarrow \infty,$
\item[($ii$)] If $\{x_n\}_{n=1}^{\infty}$ is a Cauchy sequence in  TVS-CMS $(X,p)$, then
$\{x_n\}_{n=1}^{\infty}$ is a Cauchy sequence (in usual sense) in $(X,d_p)$,
\item[($iii$)] If $(X,p)$ is a complete TVS-CMS, then $(X,d_p)$
is a complete metric space.
\end{itemize}
\label{lemma_eq_statements}
\end{lemma}

\begin{proposition}(See \cite{D_2009})
Let $(X,p)$ be a complete TVS-CMS and $T:X\rightarrow X$ satisfy the contractive
condition
\begin{equation}
p(Tx,Ty)\leq k p(x,y)
\label{contraction}
\end{equation}
for all $x,y \in X$ and $0 \leq k <1$. Then, $T$ has a unique fixed
point in $X$. Moreover, for each $x\in X$, the iterative sequence
$\{T^nx\}_{n=1}^{\infty}$ converges to fixed point.
\label{Du_thm22}
\end{proposition}

\begin{definition} \label{defn of rec TVS-cone}
Let $X$ be a nonempty set. A vector-valued function $p:X \times X \rightarrow E$ is said to be a rectangular $TVS-$ cone metric, if the following conditions hold:
\begin{itemize}
\item[($RC1$)] $0 \leq p(x,y)$ for all $x,y \in X$ and $p(x,y)=0$ if and only if $x=y$,
\item[($RC2$)]$p(x,y)=p(y,x)$ for all $x,y \in X$ ,
\item[($RC3$)]  $p(x,z)\leq p(x,y)+p(y,w)+p(w,z)$ for all $x,y \in X$ and for all distinct points $z,w \in X$ each of them different from $x$ and $y$.
\end{itemize}
The pair $(X,p)$ is then called a rectangular TVS-cone metric space (R-TVS-CMS). When $E$ is Banach space $(X,p)$ is called rectangular cone metric space (R-CMS). When $E=\mathbb{R}$ and $P=[0,\infty)$, $(X,p)$ is called rectangular  metric space (R-MS).
\end{definition}

Every TVS-CMS is R-TVS-CMS. However, the converse need not be true.
\begin{example} \label{not} ( \cite{Akbar}, see also \cite{Branciari})
Let $X=\{1,2,3,4\}$, $E=\mathbb{R}^2$ and $P=\{(x,y):x, y \geq 0\}$. Define $d:X \times X\rightarrow E$ as follows:
$$d(1,2)=d(2,1)=(3,6),~~d(2,3)=d(3,2)=d(1,3)=d(3,1)=(1,2),~~$$
$$d(1,4)=d(4,1)=d(2,4)=d(4,2)=d(3,4)=d(4,3)=(2,4).$$
Then $(X,d)$ is a R-CMS which is not a CMS, because
$$(3,6)=d(1,2)>d(1,3)+d(3,2)=(1,2)+(1,2)=(2,4).$$
\end{example}

\begin{definition} \label{conver}
Let $(X,p)$ be a rectangular TVS-cone metric space, $x \in X$ and $\{x_n\}$ a sequence in $X$.

 (i) $\{x_n\}$ is said to be a Cauchy sequence if for any $0\ll c$ there exists $n_0\in \mathbb{N}$ such that for all $m,n\in \mathbb{N}$, $n\geqslant n_0$, one has $p(x_n,x_{n+m})\ll c$.

 (ii)$\{x_n\}$ is said to converge to $x$ if for any $0\ll c$ there exists $n_0\in \mathbb{N}$ such that for all  $n\geqslant n_0$, one has $p(x_n,x)\ll c$.

 (iii) $(X,p)$ is called \textbf{complete} if every Cauchy sequence in $X$ is convergent in $X$.\\
\end{definition}

Let $T:X\rightarrow X$ be a mapping where $X$ is a R-TVS-CMS. For each $x\in X$, let
\begin{displaymath}
\textbf{O}(x)=\{x,Tx,T^2x,T^3x,\dotso\}.
\end{displaymath}

\begin{definition}
A cone metric space $X$ is said to be $T$-orbitally complete if every Cauchy sequence which is contained in $\textbf{O}(x)$ for some $x\in X$ converges in $X$.
\end{definition}

\section{Kannan's Fixed Point Theorem in R-TVS-CMS}
In order to realize the difference between TVS-CMS and R-TVS-CMS, we first prove Kannan's fixed point theorem in TVS-CMS.

\begin{theorem} \label{SA}
Let $(X,d)$ be a TVS-CMS and the mapping $T:X\rightarrow X$ satisfy the contractive condition
\begin{equation} \label{K}
d(Tx,Ty)\leqslant \beta [d(x,Tx)+d(y,Ty)]
\end{equation}
holds for all $x,y\in X$ where $\displaystyle 0<\beta<\frac{1}{2}$. If  $X$ is $T$-orbitally complete then $T$ has a unique fixed point in $X$.
\end{theorem}

\textbf{Proof}
Let $x\in X$.
\begin{displaymath}
\begin{array}{r c l}
d(Tx,T^2x)&\leqslant& \beta [d(x,Tx)+d(Tx,T^2x)]\\[3mm]
i.e.,~~d(Tx,T^2x)&\leqslant& \frac{\beta}{1-\beta}d(x,Tx)
\end{array}
\end{displaymath}
Again,
\begin{displaymath}
\begin{array}{r c l}
d(T^2x,T^3x)&\leqslant& \beta [d(Tx,T^2x)+d(T^2x,T^3x)]\\[3mm]
i.e.,~~d(T^2x,T^3x)&\leqslant& \frac{\beta}{1-\beta}d(Tx,T^2x)\leqslant {\left(\frac{\beta}{1-\beta}\right)}^2d(x,Tx)
\end{array}
\end{displaymath}
Similarly,
\begin{displaymath}
d(T^3x,T^4x)\leqslant {\left(\frac{\beta}{1-\beta}\right)}^3d(x,Tx)
\end{displaymath}
Thus in general, if $n$ is a positive integer, then
\begin{equation}
d(T^nx,T^{n+1}x)\leqslant r^nd(x,Tx)
\end{equation}
where $\displaystyle r=\frac{\beta}{1-\beta}$. Since $\displaystyle 0<\beta< \frac{1}{2}$, clearly $0<r<1$.\\

Now, our aim is to show that $\{T^nx\}$ is a Cauchy sequence.
Assume $m\in \mathbb{N}$ and $n>m$, then we have
\begin{displaymath}
\begin{array}{r c l}
d(T^nx,T^mx)&\leqslant& d(T^nx,T^{n-1}x)+d(T^{n-1}x,T^{n-2}x)+\dotso+d(T^{m+1}x,T^mx)\\[3mm]
&\leqslant& (r^{n-1}+r^{n-2}+\dotso+r^m)d(x,Tx)\\[3mm]
&\leqslant& \frac{r^m}{1-r}d(x,Tx)
\end{array}
\end{displaymath}

Let $0\ll c$ be given. Find $\delta >0$ and $q \in S$ such that $q(b)  < \delta$ implies $b\ll c$.\\
Now, since
\begin{displaymath}
\frac{r^m}{1-r}d(x,Tx) \to 0 \hspace{2 mm} as \hspace{2 mm} m\to \infty
\end{displaymath}
then find $n_0$ such that :
\begin{displaymath}
 q( {\frac{r^m}{1-r}d(x,Tx)}) < \delta  \hspace{3 mm} \forall m\geqslant n_0
\end{displaymath}
Hence, $\displaystyle \frac{r^m}{1-r}d(x,Tx)\ll c$, $\forall m\geqslant n_0$.\\

Thus, $d(T^nx,T^mx)\ll c$ for $n>m \geq n_0$. Therefore, $\{T^nx\}$ is a Cauchy sequence in $(X,d)$. Since $(X,d)$ is $T$-orbitally complete, there exists $x^*\in X$ such that $T^nx \to x^*$.\\
Choose a natural number $n_1$ such that $d(T^{n-1}x,T^nx)\ll \frac{c}{2}$ and  $\displaystyle d(T^nx,x^*)\ll \frac{c}{2}$, for all $n\geqslant n_1$. Hence, for $n>n_1$ we have
\begin{displaymath}
\begin{array}{r c l}
d(Tx^*,x^*)&\leqslant& d(TT^{n-1}x,Tx^*)+d(T^nx,x^*)\\[3mm]
&\leqslant& \beta[d(T^{n-1}x,T^nx)+d(x^*,Tx^*)]+d(T^nx,x^*)\\[3mm]
&=&\beta d(T^{n-1}x,T^nx)+\beta d(x^*,Tx^*)+d(T^nx,x^*)\\[3mm]
&\leqslant& \frac{c}{2}+ \beta d(x^*,Tx^*)+\frac{c}{2}
\end{array}
\end{displaymath}

So,
\begin{displaymath}
(1-\beta)d(Tx^*,x^*)\ll c
\end{displaymath}
Hence,
\begin{displaymath}
(1-\beta)d(Tx^*,x^*)\ll \frac{c}{m}\hspace{3 mm}\forall m\geqslant1
\end{displaymath}

Hence, $\displaystyle \frac{c}{m}-(1-\beta)d(Tx^*,x^*)\in P$, for all $m\geqslant1$. Since $\displaystyle \frac{c}{m}\to 0$ as $m\to \infty$ and $P$ is closed;
\begin{displaymath}
-(1-\beta)d(Tx^*,x^*)\in P\hspace{3 mm}and \hspace{3 mm}(1-\beta)d(Tx^*,x^*)\in P
\end{displaymath}
from the cone properties, $(1-\beta)d(Tx^*,x^*)=0$. Since $(1-\beta)$ never be equal to zero, then $d(Tx^*,x^*)=0$. Thus $Tx^*=x^*$.\\

Now, if $y^*$ is another fixed point of $T$ then $Tx^*=x^*$ and $Ty^*=y^*$. Then, we have
\begin{displaymath}
0\leqslant d(x^*,y^*)=d(Tx^*,Ty^*)\leqslant \beta d(x^*,Tx^*)+\beta d(y^*,Ty^*)=0
\end{displaymath}
Hence, $d(x^*,y^*)=0$ and so $x^*=y^*$. Therefore, the fixed point of $T$ is unique.\\

Now, we prove Kannan's fixed point theorem in R-TVS-CMS.

\begin{theorem} \label{AS}
Let $T:X\rightarrow X$ be a mapping where $(X,d)$ is a  $T$-orbitally complete R-TVS-CMS such that
\begin{equation} \label{Ka}
d(Tx,Ty)\leqslant \beta [d(x,Tx)+d(y,Ty)]
\end{equation}
holds for all $x,y\in X$ and  $\displaystyle 0<\beta< \frac{1}{2}$. Then, $T$ has a unique fixed point in $X$.
\end{theorem}
\begin{proof}
 As in the proof of Theorem \ref{SA}, for a fixed  $x\in X$, we have for all $n \in \mathbb{N}$

\begin{equation}
d(T^nx,T^{n+1}x)\leqslant r^nd(x,Tx) \label{d}
\end{equation}
where $\displaystyle r=\frac{\beta}{1-\beta}$. Since $\displaystyle 0<\beta< \frac{1}{2}$, clearly $0<r<1$.\\
Since we are not able to use the triangle inequality, we divide the proof into two cases so that we can make use of the rectangle inequality.

\textbf{Case I}:
First assume that $T^mx\neq T^nx$ for $m,n\in N,m\neq n$. Then, for $n \in N$. Clearly,
\begin{displaymath}
\begin{array}{r c l}
d(T^nx,T^{n+1}x)&\leqslant& r^nd(x,Tx)<\left(\frac{r^n}{1-r}\right)d(x,Tx)\\[3mm]
\texttt{and}~~
d(T^nx,T^{n+2}x)&\leqslant& \beta [d(T^{n-1}x,T^nx)+d(T^{n+1}x,T^{n+2}x)]\\[3mm]
&\leqslant&\beta\left[{\left(\frac{\beta}{1-\beta}\right)}^{n-1}d(x,Tx)+{\left(\frac{\beta}{1-\beta}\right)}^{n+1}d(x,Tx)\right] \textbf{\hspace{1cm}(by\hspace{2mm}\ref{d})}\\[4mm]
&\leqslant& {\left(\frac{\beta}{1-\beta}\right)}^nd(x,Tx)+{\left(\frac{\beta}{1-\beta}\right)}^{n+1}d(x,Tx)\\[4mm]
&\leqslant& \left(\frac{r^n}{1-r}\right)d(x,Tx)
\end{array}
\end{displaymath}
since $\displaystyle 0<\beta< \frac{1}{2}$, $\displaystyle \beta \leqslant \frac{\beta}{1-\beta}$.\\

Now if $m>2$ is odd then writing $m=2k+1$, $k\geqslant 1$ and using the fact that $T^px\neq T^rx$ for $p,r\in N$, $p\neq r$ we can easily show that by the rectangular inequality
\begin{equation}
\begin{array}{r c l}
d(T^nx,T^{n+m}x)&\leqslant& d(T^nx,T^{n+1}x)+d(T^{n+1}x,T^{n+2}x)+\dotso+d(T^{n+2k}x,T^{n+2k+1}x)\\[3mm]
&\leqslant& r^nd(x,Tx)+r^{n+1}d(x,Tx)+\dotso+r^{n+2k}d(x,Tx)  \textbf{\hspace{1cm}(by\hspace{2mm}\ref{d})}\\[3mm]
&\leqslant& \left(\frac{r^n}{1-r}\right)d(x,Tx) \label{e}
\end{array}
\end{equation}
Again if $m>2$ is even then writing $m=2k$, $k\geqslant2$ and using the same arguments as before we can get by the rectangular inequality
\begin{displaymath}
\begin{array}{r c l}
d(T^nx,T^{n+m}x)&\leqslant& d(T^nx,T^{n+2}x)+d(T^{n+2}x,T^{n+3}x)+\dotso+d(T^{n+2k-1}x,T^{n+2k}x)\\[3mm]
&\leqslant&[r^n+r^{n+1}+r^{n+3}+\dotso+r^{n+2k-1}]d(x,Tx) \hspace{1cm}\textbf{(by\hspace{2mm}\ref{d},\hspace{2mm}\ref{e})}
\end{array}
\end{displaymath}
Thus combining all the cases we have
\begin{equation}
d(T^nx,T^{n+m}x)\leqslant \left(\frac{r^n}{1-r}\right)d(x,Tx)\label{f}
\end{equation}
for all $m,n\in N$. Since $0<r<1$, $r^n \to 0$ as $n\to\infty $ and so by  following a similar argument as in the proof of Theorem \ref{SA}, $\{T^nx\}$ is a Cauchy sequence. Since $X$ is $T$-orbitally complete, $\{T^nx\}$ is convergent. Let $u$ is defined as:
\begin{equation}
u=\lim_{n \to \infty}{T^nx}
\end{equation}
We shall now show that $Tu=u$. Without any loss of generality we assume that $T^nx\neq u$ and $T^nx\neq Tu$ for any $n\in N$. Then by (\ref{Ka}) and the rectangular inequality, we obtain
\begin{displaymath}
\begin{array}{r c l}
d(u,Tu)&\leqslant& d(u,T^nx)+d(T^nx,T^{n+1}x)+d(T^{n+1}x,Tu)\\[3mm]
&\leqslant& d(u,T^nx)+d(T^nx,T^{n+1}x)+\beta [d(T^nx,T^{n+1}x)+d(u,Tu)]\\[3mm]
i.e., ~~d(u,Tu)&\leqslant& \frac{1}{1-\beta}[d(u,T^nx)+(1+\beta)d(T^nx,T^{n+1}x)]
\end{array}
\end{displaymath}

Since $T^nx\rightarrow u$ and $\{T^nx\}$ is Cauchy then we obtain $0\leq d(u,Tu)\ll c$ for all $c\gg 0$. Then closeness of the cone $P$ implies that $u=Tu$.\\

\textbf{Case II}:
Let $T^mx=T^nx$ for some $m,n\in N$, $m\neq n$. Let $m>n$. Then $T^{m-n}(T^nx)=T^nx$ i.e.,$T^ky=y$ where $k=m-n,y=T^nx$. Now if $k>1$
\begin{displaymath}
d(y,Ty)=d(T^ky,T^{k+1}y)\leqslant {\left(\frac{\beta}{1-\beta}\right)}^kd(y,Ty) \hspace{1cm}\textbf{(by\hspace{2 mm}\ref{d})}
\end{displaymath}
Since $\displaystyle 0<\frac{\beta}{1-\beta}<1$, $d(y,Ty)=0$ i.e., $Ty=y$. That the fixed point of $T$ is unique easily follows from (\ref{Ka}). This completes the proof of the theorem.
\end{proof}
Theorem \ref{AS} above generalizes the results obtained in \cite{Beg}, where Kannan's fixed point theorem was proved in CMS and under the normality assumption. However, the proofs in this article are done without any normality type assumption.

\section{ The nonlinear scalarization and Kannan's fixed point theorem}
In this section, we use the nonlinear scalarization function to obtain a simpler shorter proof for the Kannan's fixed point theorem in R-TVS-CMS.
\begin{theorem} \label{rectangular}
Let $(X,p)$ be a rectangular TVS-cone metric space. Then  $(X,d_p)$, where $d_p:=\xi_e \circ p$, is a rectangular metric space (R-MS).
\end{theorem}

\begin{proof}
By $RC1$, the definition of $\xi_e$ and that $P\cap -P=\{0\}$ we have $d_p(x,y)\geq 0$ for all $x, y \in X$. By $RC2$, $d_p(x,y)=d_p(y,x)$ for all $x, y \in X$. If $x=y$, then by $RC1$ $d_p(x,y)=\xi_e(0)=0$. Conversely, if $d_p(x,y)=0$, then by Lemma \ref{lemma_scalarization}, $RC1$ and that $P\cap -P=\{0\}$, we conclude that $p(x,y)=0$ and hence by $RC2$, $x=y$. Finally the rectangular inequality follows by Lemma \ref{lemma_scalarization} $(vi)$, $(vii)$ and $RC3$. \end{proof}

\begin{lemma} \label{TH}
Let $(X,p)$ be a R-TVS-CMS, $x\in X$ and $\{x_n\}_{n=1}^{\infty}$ a sequence in $X$.
Set $d_p=\xi_e \circ p$. Then the following statements hold:
\begin{itemize}
\item[($i$)]  $\{x_n\}_{n=1}^{\infty}$ converges to $x$ in the R-TVS-CMS $(X,p)$  if and only if
$d_p(x_n,x)\rightarrow 0$ as $n\rightarrow \infty,$
\item[($ii$)]  $\{x_n\}_{n=1}^{\infty}$ is Cauchy sequence in the R-TVS-CMS $(X,p)$  if and only if
$\{x_n\}_{n=1}^{\infty}$ is a Cauchy sequence in the rectangular metric space $(X,d_p)$,
\item[($iii$)]  $(X,p)$ is a complete R-TVS-CMS  if and only if $(X,d_p)$
is a complete rectangular metric space.
\end{itemize}
\end{lemma}

\begin{proof}

 Applying Theorem \ref{rectangular}, $d_p$ is a rectangular metric on $X$. Regarding (i) First, assume $\{x_n\}$ converges to $x$ in the R-TVS-CMS $(X,p)$ and let $\epsilon > 0$ be given. Find $n_0$ such that $p(x_n,x)\ll \epsilon e$ for all $n>n_0$. Therefore, by Lemma \ref{lemma_scalarization} (iv), $d_p(x_n,x)=\xi_e \circ p(x_n,x)<\epsilon $, for all $n>n_0$. Conversely, we prove that if $x_n\rightarrow x$ in $(X,d_p)$ then $x_n\rightarrow x$ in the R-TVS-CMS $(X,p)$. To this end, let $c>>0$ be given, then find $q \in S$ and $\delta >0$ such that $q(b)<\delta$ implies that $b<< c$. Since $\frac{e}{n}\rightarrow 0$ in $(E,S)$ find $\epsilon = \frac{1}{n_0}$ such that $\epsilon q(e)=q(\epsilon e)<\delta $ and hence $\epsilon e << c$. Now, find $n_0$ such that $d_p (x_n,x)=\xi_e\circ p (x_n,x)< \epsilon$ for all $n\geq n_0$. Hence, by Lemma \ref{lemma_scalarization} (iv) $p (x_n,x)<< \epsilon e<< c$ for all $n\geq n_0$.  The proof of  (ii) is similar to the proof of (i). Finally,  (iii)   is immediate from (i) and (ii).
\end{proof}
Now the proof of Theorem \ref{SA} can be achieved by Lemma \ref{TH}, Theorem \ref{rectangular} and by Kannan's fixed point theorem (see \cite{Das}) applied to the R-MS  $(X,d_p)$.

\end{document}